\documentclass[12pt]{amsart}
\usepackage{amscd}
\usepackage{graphicx}
\def\frk{\frak}               

\def\Phi{{\frk n}}
\def\Phi{{\frk N}}
\def\opn#1#2{\def#1{\operatorname{#2}}} 
\opn\chara{char} \opn\length{\ell} \opn\pd{pd} \opn\rk{rk}
\opn\projdim{proj\,dim} \opn\injdim{inj\,dim} \opn\rank{rank}
\opn\depth{depth} \opn\grade{grade} \opn\height{height}
\opn\embdim{emb\,dim} \opn\codim{codim}

\opn\Tr{Tr} \opn\bigrank{big\,rank}
\opn\superheight{superheight}\opn\lcm{lcm}
\opn\trdeg{tr\,deg}
\opn\reg{reg} \opn\lreg{lreg} \opn\ini{in} \opn\lpd{lpd}
\opn\size{size}
\opn\div{div} \opn\Div{Div} \opn\cl{cl} \opn\Cl{Cl}
\opn\Spec{Spec} \opn\Supp{Supp} \opn\supp{supp} \opn\Sing{Sing}
\opn\Ass{Ass} \opn\Min{Min}
\opn\Ann{Ann} \opn\Rad{Rad} \opn\Soc{Soc}
\opn\Im{Im} \opn\Ker{Ker} \opn\Coker{Coker} \opn\Am{Am}
\opn\Hom{Hom} \opn\Tor{Tor} \opn\Ext{Ext} \opn\End{End}
\opn\Aut{Aut} \opn\id{id}

\opn\nat{nat}
\opn\pff{pf}
\opn\Pf{Pf} \opn\GL{GL} \opn\SL{SL} \opn\mod{mod} \opn\ord{ord}
\opn\Gin{Gin} \opn\Hilb{Hilb}
\opn\aff{aff} \opn\con{conv} \opn\relint{relint} \opn\st{st}
\opn\lk{lk} \opn\cn{cn} \opn\core{core} \opn\vol{vol}
\opn\link{link} \opn\star{star}
\opn\gr{gr}

\def\pot#1#2{#1[\kern-0.28ex[#2]\kern-0.28ex]}

\opn\dirlim{\underrightarrow{\lim}}
\opn\inivlim{\underleftarrow{\lim}}
\def\Implies{\ifmmode\Longrightarrow \else
        \unskip${}\Longrightarrow{}$\ignorespaces\fi}
\def\implies{\ifmmode\Rightarrow \else
        \unskip${}\Rightarrow{}$\ignorespaces\fi}
\def\iff{\ifmmode\Longleftrightarrow \else
        \unskip${}\Longleftrightarrow{}$\ignorespaces\fi}

\let\:=\colon
\newtheorem{Theorem}{Theorem}[section]
\newtheorem{Lemma}[Theorem]{Lemma}
\newtheorem{Corollary}[Theorem]{Corollary}
\newtheorem{Proposition}[Theorem]{Proposition}
\newtheorem{Remark}[Theorem]{Remark}

\newtheorem{Definition}[Theorem]{Definition}

\let\epsilon\varepsilon
\let\phi=\varphi
\let\kappa=\varkappa
\def\qed{\ifhmode\textqed\fi
      \ifmmode\ifinner\quad\qedsymbol\else\dispqed\fi\fi}
\def\textqed{\unskip\nobreak\penalty50
       \hskip2em\hbox{}\nobreak\hfil\qedsymbol
       \parfillskip=0pt \finalhyphendemerits=0}
\def\dispqed{\rlap{\qquad\qedsymbol}}

\opn\dis{dis}
\def\pnt{{\raise0.5mm\hbox{\large\bf.}}}

\opn\Lex{Lex}

\begin{document}
\title{Spanning Simplicial Complexes of Uni-Cyclic Graphs}

\author{Imran Anwar$^1$,\ \ Zahid Raza$^2$, Agha Kashif$^2$ }
 \thanks{ {\bf 1.} COMSATS Institute of Information Technology Lahore,
         Pakistan.\\ {\bf 2.} National University of Computer and Emerging Sciences Lahore Campus, Pakistan }
\email { imrananwar@ciitlahore.edu.pk,\ zahid.raza@nu.edu.pk,\
ajeekas12@hotmail.com}
 \maketitle
\begin{abstract}
In this paper, we introduce the concept of {\em spanning simplicial
complexes $\Delta_s(G)$} associated to a simple finite connected
graph  $G$. We give the characterization of all spanning trees of
the {\em uni-cyclic graph} $U_{n,m}$. In particular, we give the
formula for computing the Hilbert series and $h$-vector of the
Stanley Riesner ring $k\big[\Delta_s(U_{n,m})\big]$. Finally, we
prove that the {\em spanning simplicial complex} $\Delta_s(U_{n,m})$
is shifted hence $\Delta_s(U_{n,m})$ is shellable.

\end{abstract}

\section{introduction}
Suppose $G(V,E)$ is a finite simple connected graph with the vertex
set $V$ and edge-set $E$, a spanning tree of a simple connected
finite graph $G$ is a subgraph of $G$ that contains every vertex of
$G$ and is also a tree. We represent the edge-set of all spanning
trees of a graph $G$ by $s(G)$. In this paper, for a finite simple
connected graph $G(V,E)$, we introduce the concept of {\em spanning
simplicial complexes} by associating a simplicial complex
$\Delta_s(G)$ defined on the edge set $E$  of the graph $G$ as
follows:
$$\Delta_s(G)=\big\langle F_i \mid F_i \in s(G)\big\rangle$$

It is always possible to associate $\Delta_s(G)$ to any simple
finite connected graph $G(V,E)$ but the characterization of $s(G)$
has been a problem in this regard.

For the {\em uni-cyclic graphs $U_{n,m}$}, we prove some algebraic
and combinatorial properties of {\em spanning simplicial complex}
$\Delta_s(U_{n,m})$. Where, a uni-cyclic graph $U_{n,m}$ is a
connected graph over $n$ vertices and containing exactly one cycle
of length $m$. In Proposition \ref{scn}, we give the
characterization of $s(U_{n,m})$.  Moreover, we give
characterizations of the $f-vector$ and $h-vector$ in Lemma
\ref{fsc} and Theorem \ref{hch} respectively, which enable us to
device a formula to compute the {\em Hilbert series} of the Stanley
Reisner ring $k\big[\Delta_s(U_{n,m})\big]$ in Theorem \ref{Hil}. In
the Theorem \ref{shift}, we show that the {\em spanning simplicial
complex} $\Delta_s(U_{n,m})$ is {\em shifted}. So, we have the
corollary \ref{shell} that the {\em spanning simplicial complex}
$\Delta_s(U_{n,m})$ is {\em shellable}.
\section{Basic Setup}
In this section, we give some basic definitions and notations which
we will follow in this paper.
\begin{Definition}\label{spa}
{\em  A spanning tree of a simple connected finite graph $G(V,E)$ is
a subtree of $G$ that contains every vertex of $G$.

We represent the collection of all edge-sets of the spanning trees
of $G$ by $s(G)$, in other words;
$$s(G)=\{E(T_i)\subset E ,    \hbox {\, where $T_i$ is a spanning tree of $G$}  \}.$$
}
\end{Definition}
\begin{Remark}\label{spe}
{\em It is well known that for any simple finite connected graph
spanning tree always exist. One can find a spanning tree
systematically by {\em cutting-down method}, which says that
spanning tree of a given simple finite connected graph is obtained
by removing one edge from each cycle appearing in the graph.

For example by using {\em cutting-down method} for the graph given
in figure 1 we obtain:
$$s(G)=\{ \{ e_2, e_3, e_4\} ,\{e_1, e_3, e_4\}, \{e_1, e_2, e_4\}, \{e_1, e_2, e_3\} \}$$
} \end{Remark}


\begin{center}
\begin{picture}(300,80)\label{fig}
\thicklines \put(100,27){\line(1,0){69.6}}\put(100,60){\line(1,0){70}}
\put(97,25){$\bullet$}\put(97,18){${v_{1}}$}
\put(99,27){\line(0,1){33}}\put(97,57){$\bullet$}\put(97,64){$v_{4}$} \put(88,41){$e_{4}$}\put(171,41){$e_{2}$}
\put(132,18){${e_{1}}$}
\put(132,64){$e_{3}$}
\put(165.5,25){$\bullet$}\put(166,18){${v_{2}}$}
\put(168,27){\line(0,1){33}}\put(165.5,57){$\bullet$}\put(165,64){$v_{3}$}
\put(116,-4){\tiny Fig. 1 . $C_4$}
\end{picture}
\end{center}

\begin{Definition}{\em
A Simplicial complex $\Delta$ over a finite set
$[n]=\{1, 2,\ldots,n \}$ is a collection of subsets of $[n]$, with
the property that $\{i\}\in \Delta$ for all $i\in[n]$, and if $F\in
\Delta$  then $\Delta$ will contain all the subsets of $F$
(including the empty set). An element of $\Delta$ is called a face
of $\Delta$, and the dimension of a face $F$ of $\Delta$ is defined
as $|F|-1$, where $|F|$ is the number of vertices of $F$. The
maximal faces of $\Delta$ under inclusion are called facets of
$\Delta$. The dimension of the simplicial complex $\Delta$ is :
$$\hbox{dim} \Delta = \max\{\hbox{dim} F | F \in \Delta\}.$$
We denote the simplicial complex $\Delta$ with facets $\{F_1,\ldots
, F_q\}$ by $$\Delta = \big\langle F_1,\ldots, F_q\big\rangle $$ }
\end{Definition}
\begin{Definition}\label{fvec}{\em
For a simplicial complex $\Delta$ having dimension $d$, its
$f-vector$ is a $d+1$-tuple, defined as:
$$f(\Delta)=(f_0,f_1,\ldots,f_d)$$
where $f_i$ denotes the number of $i-dimensional$ faces of $\Delta.$
}\end{Definition}

\begin{Definition}\label{co}{\bf (Spanning Simplicial Complex )}\\
{\em
For a simple finite connected graph $G(V,E)$ with $s(G)=\{E_1,
E_2,\ldots,E_s\}$ be the edge-set of all possible spanning trees of
$G(V,E)$, we define a simplicial complex $\Delta_s(G)$ on $E$ such
that the facets of $\Delta_s(G)$ are precisely the elements of
$s(G)$, we call $\Delta_s(G)$ as the {\em spanning simplicial
complex} of $G(V,E)$. In other words;
$$\Delta_s(G)=\big\langle E_1,E_2,\ldots,E_s\big\rangle.$$
}\end{Definition}
For example; the spanning simplicial complex of
the graph $G$ given in figure 1 is:
$$\Delta_s(G)=\big\langle\{ e_2, e_3, e_4\} ,\{e_1, e_3, e_4\}, \{e_1, e_2, e_4\}, \{e_1, e_2, e_3\}\big\rangle$$

We conclude this section with the definition of {\em uni-cyclic
graph} $U_{n,m}$;
\begin{Definition}{\em
A {\em uni-cyclic graph} $U_{n,m}$ is a connected graph on $n$
vertices, and containing exactly one cycle  of length $m$ (with
$m\leq n$)}.
\end{Definition}
The number of vertices in $U_{n,m}$ equals the number of
edges. In particular, if $m=n$ then $U_{n,m}$ is simply $n$-cyclic
graph.

\section{Spanning trees of $U_{n,m}$ and Stanley-Reisner ring $\Delta_s(U_{n,m})$ }

Throughout the paper, we fix the edge-labeling $\{e_1, e_2, \ldots,
e_m,e_{m+1},\ldots,e_n\}$ of $U_{n,m}$ such that $\{e_1, e_2, \ldots,
e_m\}$ is the edge-set of the only cycle in $U_{n,m}$. In the
following result, we give the characterization of $s(U_{n,m})$.

\begin{Lemma}\label{scn}{\bf Characterization of $s(U_{n,m})$}\\  {\em
Let $U_{n,m}$ be the {\em uni-cyclic graph} with the edge set
$E=\{e_1, e_2,\ldots , e_n\}$. A subset $E(T_i)\subset E$ will belong to
$s(U_{n,m})$ if and only if $T_i=E\setminus \{e_i\}$ for some
$i\in\{1,\ldots,m\}$. In particular;
$$s(U_{n,m})=\{\hat{E_i}\, \mid\, \hat{E_i}=E\setminus \{e_i\} \hbox{\ for all\, }1\leq i\leq m \}$$

}\end{Lemma}
\begin{proof}
As $U_{n,m}$ contains only one cycle of $m$ vertices, so its
spanning trees will be obtained by just removing one edge from the
cycle of $U_{n,m}$ follows from \ref{spe}. Which implies that
$$s(U_{n,m})=\{\hat{E_i}\, \mid\, \hat{E_i}=E\setminus \{e_i\} \hbox{\ for all\, }1\leq i\leq m \}$$
\end{proof}
We need the following elementary proposition in order to prove our
next result.

\begin{Proposition}\label{fd}
  \em {For a simplicial complex $\Delta$ over $[n]$ of dimension $d$, if $f_t$ = ${n} \choose t+1$ for some $t\leq d$
  then $f_i =   {{n}\choose {i+1}}$ for all $0 \leq i < t$. }
  \end{Proposition}
  \begin{proof}
Suppose $\Delta$ be any simplicial complex over $[n]$ with dimension
$d$ having\linebreak $f_t={n\choose t+1}$ for some $t\leq d$. It
implies that $\Delta$ will contain all the subset of $[n]$ with the
cardinality $t+1$ (which is $f_t={n\choose t+1}$), then it is
sufficient to prove that $\Delta$ will contain every subset of $[n]$
with the cardinality $\mid i\mid$ with $i\leq t$. Let us take any
arbitrary subset $F$ of $[n]$ with  $\mid F\mid <t+1$, then by
adding more vertices to $F$ we can extend $F$ to $\tilde{F}$ with
$\mid \tilde{F} \mid = t+1$, which is already in $\Delta$ therefore
the assertion follows immediately from the definition of simplicial
complex. Hence $\Delta$ will contain all the subsets of $[n]$ with
the cardinality $\leq t$, that is
$$f_i={{n}\choose {i+1}} \hbox{\, for all\, } 0 \leq i < t.$$
   \end{proof}
Our next result is the characterization of the $f$-vector of
$\Delta_s(U_{n,m})$.
\begin{Proposition}\label{fsc}
  \em{Let $\Delta_s(U_{n,m})$ be the spanning simplicial complex of {\em uni-cyclic graph $U_{n,m}$}, then $dim(\Delta_s(U_{n,m})=n-2$ and having the following
  f-vector\linebreak
     $ f(\Delta_s(U_{n,m})= \big(f_0, f_1,\ldots, f_{n-2})\hbox{\, with\, }$
     $$f_{i}=\left\{
               \begin{array}{ll}
                 {n\choose{i+1}}, & \hbox{for}\,\,\,  i\leq m-2;\\\\
                 {n\choose i+1}-{{n-m}\choose{i-m+1}},& \hbox{for}\,\,  m-2<i\leq n-2.
               \end{array}
             \right.$$
  }\end{Proposition}
  \begin{proof}
  Let $E=\{e_1,e_2,\ldots,e_n\}$ be the set of edges of $U_{n,m}$, then
  from \ref{scn} ;
  $$s(U_{n,m})=\{\hat{E_i}\, \mid\, E_i=E\setminus \{e_i\} \hbox{\ for all\, }1\leq i\leq m \}.$$

  Therefore, by definition \ref{co} we have;
  $$\Delta_s(U_{n,m})=\big\langle \hat{E_1},\hat{E_2},\ldots,\hat{E_m}\big\rangle.$$

Since each facet $\hat{E_i}$ is of the same dimension $n-2$ ( as
$\mid \hat{E_i}\mid=n-1$), therefore $\Delta_s(U_{n,m})$ will be of
dimension $n-2$. Also, it is clear from the definition of
$\Delta_s(U_{n,m})$ that $\Delta_s(U_{n,m})$ contains all those
subsets of $E$ that do not contain $\{e_1,\ldots,e_m\}$.

Let us take any arbitrary subset $F\subset E$ consisting of $m-1$
members. As $F$ consists of $m-1$ elements, then it is clear that
$\{e_1,\ldots,e_m\}$ can not appear in $F$, therefor, $F\in
\Delta_s(U_{n,m})$. It follows that $\Delta_s(U_{n,m})$ contains all
possible subsets of $E$ with the cardinality $m-1$, therefore,
$f_{m-2}={n\choose m-1}$. Hence from \ref{fd}, we have
$f_i={n\choose{i+1}}$ for all
$i\leq m-2$.

In order to prove the other case, we need to compute all the subsets
of $F\subset E$ with $\mid F\mid=i(\geq m)$ containing the cycle
$\{e_1,\ldots,e_m\}$. We have in total $n$ elements in $E$ and we
are choosing $i$-elements out of it with the condition that
$\{e_1,\ldots, e_m\}$ will be a part of it. By using the inclusion
exclusion principle, we get that there are ${n-m}\choose{i-m+1}$
subsets of $E$ consisting of $i+1\, (\geq m)$ elements and
containing the cycle $\{e_1, \ldots,e_m\}$. In total, we have
$n\choose{i+1}$ subsets of $E$ with the cardinality $i+1$,
therefore, we have the $f_i={n\choose i+1}-{{n-m}\choose{i-m}+1}$
for $m-2< i\leq n-2$.
\end{proof}

For a simplicial complex $\Delta$ over $[n]$, one would associate to it the
Stanley-Reisner ideal, that is, the monomial ideal $I_{\mathcal N}(\Delta)$ in
$S=k[x_1, x_2,\ldots ,x_n]$ generated by monomials corresponding to
non-faces of this complex (here we are assigning one variable of the
polynomial ring to each vertex of the complex). It is well known
that the Stanley-Reisner ring $k[\Delta]=S/I_{\mathcal N}(\Delta)$
is a standard graded algebra. We refer the readers to \cite{HP} and
\cite{Vi} for more details about graded algebra $A$, the Hilbert
function $H(A,t)$ and the Hilbert series $H_t(A)$ of a graded algebra.
\begin{Definition}\label{hvec}{\em
Let $A$ be a standard graded algebra and
$$h(t)=h_0+h_1t+\cdots+h_rt^r$$ the (unique) polynomial with
integral  coefficients such that $h(1)\neq 0$ and satisfying
$$H_t(A)=\frac{h(t)}{(1-t)^d}$$
where $d=\mbox{dim}(A)$. The $h$-vector of $A$ is defined by
$h(A)=(h_0,\ldots,h_r)$. }
\end{Definition}
 Now we give the formula for the $h$-vector of $k\big[\Delta_s(U_{n,m})\big]$;

\begin{Theorem}\label{hch} {\em
If $\Delta_s(U_{n,m})$ is a spanning simplicial complex of the
uni-cyclic graph $U_{n,m}$ and $(h_i)$ is the $h$-vector of
$k\big[\Delta_s(U_{n,m})\big]$, then $h_k=0$ for $k>n-1$ and
$$h_k=\left\{
        \begin{array}{ll}
          \sum\limits_{i=0}^{k} (-1)^{k-i}{n-1-i\choose k-i}{n\choose i}, & \hbox{for}\,\, k\leq m-1; \\
          \sum\limits_{i=0}^{k} (-1)^{k-i}{n-1-i\choose k-i}\big[{n\choose i}-{{n-m}\choose{i-m}}\big], &
\hbox{for}\,\, m-1<k\leq n-1.
        \end{array}
      \right.
$$}
\end{Theorem}
\begin{proof}
We know from \cite{Vi} that, if $\Delta$ be any simplicial complex
of dimension $d$ and $(h_i)$ be the $h$-vector of $k[\Delta]$, then
$h_k=0$ for $k>d+1$ and
$$h_k=\sum_{i=0}^{k} (-1)^{k-i}{n-1-i\choose k-i}{f_{i-1}}\hbox{ \, for \, } 0\leq k\leq d+1$$
The result follows by substituting the values $f_i$'s ( from
\ref{fsc}) in the above formula.
\end{proof}

One of our main results of this section is as follows;
\begin{Theorem}\label{Hil} {\em Let $\Delta_s(U_{n,m}) $ be the spanning simplicial complex of
$U_{n,m}$, then the Hilbert series of the Stanley-Reisner ring
$k\big[\Delta_s(U_{n,m})\big]$ is given by,
$$H(k[\Delta_s(U_{n,m})],t)=
\sum_{i=0}^{m-2}\frac{{n\choose
{i+1}}{t^{i+1}}}{(1-t)^{i+1}}+\sum_{i=m-1}^{n-2}\frac{\big[{{n\choose
i+1}-{{n-m}\choose{i-m+1}}}\big]{t^{i+1}}}{(1-t)^{i+1}}+1.$$ }
\end{Theorem}

\begin{proof}
We know from \cite{Vi} that if $\Delta$ be any simplicial complex of
dimension $d$ with\linebreak $f(\Delta)=(f_0, f_1, \ldots,f_d)$ be
its $f$-vector, then the Hilbert series of Stanley-Reisner ring
$k[\Delta]$ is given by
$$H(k[\Delta],t)=
\sum_{i=0}^{d}\frac{f_i t^{i+1}}{(1-t)^{i+1}}+1.$$ The result
immediately follows by substituting the values of $f_i$'s in the
above formula from \ref{fsc}.
\end{proof}

Algebraic shifting theory was introduced by G. Kalai in \cite{K}, and it describes strong tolls to investigate
one of the most interesting and powerful property of simplicial complexes.
\begin{Definition}{\em  A simplicial
complex $\Delta$ on $[n]$ is {\em shifted} if, for $F\in \Delta$,
$i\in F$ and $j\in [n]$ with $j>i$, one has
$(F\setminus\{i\})\cup\{j\}\in \Delta$.}
\end{Definition}

\begin{Theorem}\label{shift}
{\em  The spanning simplicial complex $\Delta_s(U(m,n))$ of the {\em
uni-cyclic graph} is  {\em shifted}.

}\end{Theorem}
\begin{proof}
From \ref{co} and \ref{scn}, we know that
$$\Delta_s(U_{n,m})=\big\langle{\hat{E_1}},{\hat{E_2}},\ldots, \hat{{E_m}}\big\rangle.$$
It is sufficient to prove the shifted condition on the facets of the
simplicial\linebreak complex. For some facet ${\hat{E_i}}\in
\Delta_s(U_{n,m})$ with $j\in {\hat{E_i}}$, we claim that\linebreak
$({\hat{E_i}}\setminus \{j\})\cup\{k\}\in \{\hat{E_1}, \hat{E_2},
\ldots, \hat{E_m}\}$ with $k\not\in{\hat{E_i}} $ and $j<k$. By the
definition of  $\hat{E_i}$, we have only one possibility for $k$
that is $k=i$ and $j<i\leq m$, therefore, it is easy to see that
$({\hat{E_i}}\setminus \{j\})\cup\{i\}=\hat{E_i}$ for all $j(<i)\in
\hat{E_i}$. Hence $\Delta_s(U_{n,m})$ is a shifted simplicial
complex.
\end{proof}
The above theorem immediately implies the following result;
\begin{Corollary}\label{shell}
{\em The spanning simplicial complex $\Delta_s(U(m,n))$ is
shellable.}
\end{Corollary}

\end{document}